\documentclass[11pt]{article}
\usepackage{mathrsfs}
\usepackage{amsmath}
\usepackage{amsfonts}
\usepackage{amssymb}
\usepackage{amsthm}
\newtheorem{theorem}{Theorem}[section]
\newtheorem{definition}[theorem]{Definition}

\newtheorem{cor}[theorem]{Corollary}

\newtheorem{example}[theorem]{Example}
\newtheorem{remark}[theorem]{Remark}
\title{\Large\bf The strong b-lattice of quasi skew-rings}
\vspace{-3em}
\author{\bf Rituparna Ghosh }
\date{}

\begin{document}

\maketitle

\vspace{-3em}
\begin{center}
{\small\it Department of Mathematics, Vivekananda Mahavidyalaya,}\\
{\small\it Haripal, Hooghly 712405, India.}\\
{\small e-mail: rituparnaghosh@vmharipal.ac.in }
\end{center}

\begin{abstract}
\noindent
The notions of GV-inverse semigroups and completely $\pi$-inverse semigroups accord when extended under semirings to quasi completely inverse semirings. A semiring is quasi completely inverse semiring if and only if it is b-lattice of quasi skew-rings. A semiring $S$ is strongly additively quasi completely inverse semiring
if $S$ is quasi completely regular and strongly additively quasi inverse semiring. In this paper we study the relation between quasi completely inverse semirings and strongly additively quasi completely inverse semirings and give the necessary and sufficient conditions for a strongly additively quasi completely inverse semiring to be a strong b-lattice of quasi skew-rings.
\end{abstract}

\vspace{.5em}
AMS Mathematics Subject Classification (2010):
16Y60, 20M10, 20M07.

{\bf Key Words: }{quasi completely regular semirings, quasi completely inverse semirings, strongly additively quasi completely inverse semiring, generalized Clifford semiring, strong b-lattice of semirings.}

\section{Introduction}
In 1934, Vandiver averred
the idea of associative algebra or, semiring as a set of elements forming both a semigroup
under addition and under multiplication where the right and left hand
distributive laws hold. Semigroup theory is
considered to be the constituent of semiring theory. Various semigroup
structures and properties have been extended analogously in semiring
theory. In 1941, A. H. Clifford \cite{cliffrd1} established that a semigroup contains relative inverses if and only if it is a class sum of mutually disjoint groups and also if a semigroup admits relative inverses then it can be characterised as semilattices
of completely simple semigroups. This became the onset of the study of semilattice decompositions of semigroups and the bedrock for the structural
investigation of completely regular semigroups. In \cite{PR}, M. Petrich and N.R. Reilly  characterized completely regular semigroups as a union of groups and also as 
semilattice of completely simple semigroups. They also studied special classes of completely regular semigroups called Clifford semigroups. Clifford semigroups 
can be expressed as a strong semilattice of groups \cite{howie}. Such semigroups are considered as generalization of groups. In the year 2006, Sen, Maity and Shum \cite{maity5} extended the idea of completely regular semigroups in semiring theory and proved that a semiring S is completely regular if
and only if it is union of skew-rings if and only if it is a b-lattice of completely simple
semirings. The authors also studied generalised Clifford semirings \cite{sen1}  as those completely regular semirings
$S$ which are additive inverse semirings where the set of all additive idempotent elements of $S$, denoted by $E^{+}(S)$, are k-ideals of $S$. It was determined that generalised Clifford semirings are strong b-lattice of skew-rings.

Many authors like M. L. Veronesi, S. Bogdanovic and M. Ciric studied decomposition of completely $\pi$-regular semigroups into
semilattice of Archimedean semigroups. It was established that S is completely Archimedean
if and only if it is a nil-extension of completely simple semigroups if and only if it is
Archimedean and completely $\pi$-regular \cite{bogda1}. It was established that a GV - semigroup is $\pi$-regular and every $\mathcal{H^{*}}$-class is a $\pi$-group, where a $\pi$-group is nil-extension of a group. A semigroup S is a GV-inverse semigroup if S is a
GV-semigroup and every regular element of S possesses a unique inverse. A semigroup
is GV-inverse if and only if it is a semilattice of $\pi$-groups. Recently in 2018, J. Zhang, Y. Yang, and S. Ran \cite{zhangmain} established the necessary and sufficient conditions for a GV-inverse semigroup to be a strong semilattice of $\pi$-groups.

Motivated by these research literatures, in \cite{maity1} we extended the structure theorem of completely $\pi$- regular semigroups and GV-semigroups (semigroups
of Galbiati-Veronesi) in the domain of semirings and called them quasi completely regular semirings.
In \cite{maity1} quasi completely regular semirings have been characterized not only as (disjoint) union of quasi skew-rings but also as idempotent semiring of quasi skew-rings and b-lattice of completely Archimedean semirings. Moreover it was proved in \cite{maity1} that every additively regular element of a quasi completely regular semiring is completely regular, implying the most important finding here that when
extended over semirings, completely $\pi$- regular semigroup and GV-semigroup structures are no more different. In the subsequent studies several results and congruences \cite{maity2, maity3} were established on quasi completely regular semirings which includes nil-extension of completely simple semirings where nil-extension on semirings were defined with the
help of bi-ideals. In \cite{maity4} we initiated the notions of quasi completely inverse semiring and strongly additively quasi completely inverse semiring and also  characterized a quasi completely inverse semiring to be b-lattice of quasi skew-rings.

In this paper we further relate the notions of strongly additively quasi completely inverse semiring with quasi completely inverse semiring and give the necessary and sufficient conditions for a strongly additively quasi completely inverse semiring to be a strong b-lattice of quasi skew-rings. The priliminaries required for the study are discussed in Section $2$. In Section $3$ the ideas and results of quasi completely regular semirings and quasi completely inverse semirings are recollected. The important concepts and main result of the paper are discussed in the Sections $4$ and $5$.
 
\noindent

\section{Priliminaries}

A semiring $(S, +, \cdot)$ is a type (2, 2) algebra whose
semigroup reducts $(S, +)$ and $(S, \cdot)$ are connected by ring
like distributivity, that is, $a(b+c)=ab+ac$ \, and
$(b+c)a=ba+ca$ for all $a, b, c \in S$. $S$ is not considered to be commutative. 

Let $(S, +, \cdot)$ be a semiring. An element $a$ in $S$ is: 

\begin{itemize}
    \item [-] \emph{additively regular} if there exists an element $x\in S$
such that $a+x+a=a$;
    \item [-] \emph{additively completely regular} if there exists an element $x\in S$
such that $a+x+a=a$ and $a+x=x+a$; if such an element $x$ exists then it is unique and satisfies $x=x+a+x$;
    \item[-] \emph{completely regular} \cite{maity5} if there exists $x\in S$
such that $a=a+x+a$, $a+x=x+a$ and $a(a+x)=a+x$;
    \item[-] \emph{additively quasi regular} if there exists a
positive integer $n$ such that the element
$na=\underbrace{a+a+\dots+a}_{\text{$n$ times}}$ is additively
regular;
    \item[-] \emph{additively quasi completely regular}
 if there exists a positive integer $n$ such that
$na$ is additively completely regular, that is, there exists an element $x
\in S$ such that $na$ = $na+x+na$ and $na+x$ = $x+na$; if such an element $x$ exists then it is unique and satisfies $x=x+na+x$; 
    \item[-] \emph{quasi completely regular}
\cite{maity1} if there exists a positive integer $n$ such that
$na$ is completely regular, that is, there exists an element $x
\in S$ such that $na$ = $na+x+na$, $na+x$ = $x+na$, and
$na(na+x)$  = $na+x$. 
\end{itemize}

A semiring $(S, +, \cdot)$ is a
\emph{skew-ring} \cite{maity5} if its additive reduct $(S,+)$ is a group, but not
necessarily abelian group.
If for every $a\in S$, there exists a positive integer $n$ such that $na$ lies in a subskew-ring $R$ of $S$, then
$S$ is said to be a \emph{quasi skew-ring} \cite{maity1}. 
A semigroup is called an \emph{inverse semigroup} if for every $a \in S$
there exists a unique element $a'\in S$ such that $aa'a=a$ and
$a'aa'=a'$. A semigroup is
\emph{$\pi$-inverse semigroup} if for every $a\in S$ there exists a
positive integer $m$ and a unique $x\in S$ such that $a^m=a^m x
a^m$ and $x=x a^m x$. A semigroup $S$ is \emph{strongly $\pi$-inverse} semigroup 
if $S$ is $\pi$-regular and the idempotent elements in $S$ commute. In \cite{PR}, the authors defined a special case of semilattice of semigroups in which the multiplication is defined by means of homomorphisms called strong semilattice of semigroups. We recall that definition as follows:

Let $Y$ be a semilattice and for each $\alpha \in Y$, let $S_{\alpha}$ be a semigroup, $S_{\alpha} \cap S_{\beta}$ = $\emptyset$ if $\alpha \neq \beta$. For each $\alpha, \beta \in Y$ with $\alpha \geq \beta$, let $\chi_{_{\alpha, \beta}}: S_{\alpha} \longrightarrow S_{\beta} $ be a homomorphism such that

(i) $\chi_{_{\alpha, \alpha}}$ = $I_{_{S_{\alpha}}}$;

(ii) $\chi_{_{\alpha, \beta}}$$\chi_{_{\beta, \gamma}}$ = $\chi_{_{\alpha, \gamma}}$ if $\alpha \geq \beta \geq \gamma$.

On $S$ = $\displaystyle{\bigcup_{\alpha \in Y}S_\alpha}$, define a multiplication by

\begin{center}
    $a \star b$ = $(a \chi_{_{\alpha, \alpha \beta}})$$(b\chi_{_{\beta, \alpha\beta}})$ \,\,\, $(a \in S_{\alpha}, b \in S_{\beta})$.
\end{center}
With this multiplication, $S$ forms a \emph{strong semilattice $Y$ of semigroups} $S_{\alpha}$ and is denoted by  $S=[Y; S_\alpha, \chi_{_{\alpha,\beta}}]$. In \cite{sen1} the authors have given a structure {in semirings which is analogous to the structure of strong semilattice of semigroups where they have used b-lattice \cite{maity5} instead of semilattice. They called it \emph{strong b-lattice of semirings}. The definition has been recalled in Section $5$.

A semiring $(S, +, \cdot)$ is :

\begin{itemize}
    \item[-] an
\emph{additive inverse semiring} if its additive reduct $(S, +)$ is an
inverse semigroup;
    \item[-] an \emph{additively quasi
inverse semiring} if its additive reduct $(S, +)$ is a $\pi$-inverse semigroup, i.e.,
for every $a \in S$ there exists a positive integer $n$ and a unique $x
\in S$ such that, $na+x+na$=$na$ and $x+na+x$=$x$;
    \item[-] a \emph{quasi completely inverse semiring} \cite{maity4} if $S$ is quasi completely
regular semiring and additively quasi inverse semiring.
\end{itemize}
 
 A semiring $(S, +, \cdot)$ is a \emph{b-lattice} \cite{maity5} if
$(S, \cdot)$ is a band and $(S, +)$ is a semilattice. An \emph{ideal} $I$ of a semiring $S$ is a non-empty subset of $S$ such that for all $a,b \in I$ and $x \in S$, $a+b, \,\,xa,\,\, ax \in I$. An ideal $I$ of a semiring $S$ is a \emph{k-ideal} \cite{sen1} of $S$ if $a \in I$ and either $a+x \in I$ or $x+a \in I$ for some $x \in S$ implies $x \in I$. A non-empty subset $I$ of semiring $S$ is
said to be a \emph{bi-ideal} \cite{monico} of a semiring $S$ if $a\in I$ and $x\in S$ imply that
$a+x,\,\, x+a,\,\, ax,\,\, xa\, \in I$. 
Let $I$ be a
bi-ideal of a semiring $S$. We define a relation $\rho_{_I}$ on
$S$ by $a \rho_{_I} b$ if and only if either $a, b \in I$ or $a =
b$ where $a, b \in S$. It is easy to verify that $\rho_{_I}$ is a
congruence on $S$. This congruence is said to be \emph{Rees congruence}
on $S$ and the quotient semiring $S/\rho_{_I}$ contains a zero,
namely $I$. This quotient semiring $S/\rho_{_I}$ is said to be
the \emph{Rees quotient semiring} \cite{maity2} and is denoted by $S/I$. In this case
the semiring $S$ is said to be an \emph{ideal extension} or simply an
extension of  $I$ by the semiring $S/I$. An ideal extension $S$
of  a semiring $I$ is a \emph{nil-extension} \cite{maity2} of $I$ if for any $a \in S$
there exists a positive integer $n$ such that $na \in I$.

Throughout this paper, certain notations shall be followed. $E^+(S)$ will be the set of all
additive idempotents of the semiring $S$, the set
of all additive inverses of $a$ (if it exists) in a semiring $S$
will be denoted by $V^+(a)$, the set of all
additively regular elements of a semiring $S$ will be denoted by $Reg^+(S)$ and the set of all completely regular elements in the semiring $S$ will be denoted by $Cr(S)$ . The
identity relation and the universal relation on $S$ are denoted by
$\epsilon_{S}$ and $\omega_{S}$, respectively or, simply by
$\epsilon$ and $\omega$ when no confusion arises.

A congruence $\rho$ on a semigroup $S$ is called \emph{idempotent-separating} \cite{PR} if each $\rho$-class contains atmost one idempotent, that is, ${\rho}{\big{|}_{E(S)}}=\epsilon$. A congruence $\rho$ on a semiring $S$ is called \emph{additive idempotent-separating} if each $\rho$-class contains atmost one additive idempotent, that is, ${\rho}{\big{|}_{E^{+}(S)}}=\epsilon$.
A congruence $\rho$
on a semiring $S$ is called a \emph{b-lattice congruence (idempotent
semiring congruence)} if $S/\rho$ is a b-lattice (respectively, an
idempotent semiring). A semiring $S$ is called a b-lattice
(idempotent semiring ) $Y$ of semirings $S_{\alpha} ( \alpha \in
Y) $ if $S$ admits  a b-lattice congruence (respectively, an
idempotent semiring congruence) $\rho$ on $S$ such that $Y = S/
\rho$ and each $S_{\alpha}$ is a $\rho$-class mapped onto
$\alpha$ by the natural epimorphism $ {\rho}^{\#}\,\,:\,S
\longrightarrow Y$. We write $S$ = $(Y;S_{\alpha})$ in such a case.

As usual, the Green's relations on the
semiring $(S, +, \cdot)$ are denoted by by $\mathscr{L}$, $\mathscr{R}$,
$\mathscr{D}$, $\mathscr{J}$ and $\mathscr{H}$ and
correspondingly, the $\mathscr{L}$-relation,
$\mathscr{R}$-relation, $\mathscr{D}$-relation,
$\mathscr{J}$-relation and $\mathscr{H}$-relation on $(S, +)$ are
denoted by $\mathscr{L}^+$, $\mathscr{R}^+$, $\mathscr{D}^+$,
$\mathscr{J}^+$ and $\mathscr{H}^+$, respectively. For any $a
\in S$, we let $H^+_a$ be the $\mathscr{H}^+$-class in $(S, +)$
containing $a$. If $(S, +,
\cdot)$ be an additively quasi regular semiring, then the
relations $\mathscr{L}^{*+}, \mathscr{R}^{*+}, \mathscr{J}^{*+},
\mathscr{H}^{*+}$ and $\mathscr{D}^{*+}$ are defined by
\vspace{-.7em}
\begin{center}
$a\, \mathscr{L}^{*+} \, b$ if and only if  $pa \,
\mathscr{L}^{+} \, qb$

$a\, \mathscr{R}^{*+} \, b$ if and only if $pa \, \mathscr{R}^{+}
\, qb$

$a\, \mathscr{J}^{*+} \, b$ if and only if $pa \, \mathscr{J}^{+}
\, qb$

$\mathscr{H}^{*+} = \mathscr{L}^{*+} \cap \mathscr{R}^{*+}$ \, \,
\, \, and\,\, \, \, \,   $\mathscr{D}^{*+} = \mathscr{L}^{*+} \, o\,
\mathscr{R}^{*+}$
\end{center}
\vspace{-1em} where $p$ and $q$ are the smallest positive integers
such that $pa$ and $qb$  are additively regular.

A completely regular semiring $(S, +, \cdot)$ is said to be
\emph{completely simple} \cite{maity5} if any two elements of $S$ are
$\mathcal{J}^{+}$- related.
A quasi
completely regular semiring $(S, +, \cdot)$ is said to be
\emph{completely Archimedean} \cite{maity1} if any two elements of $S$ are
$\mathcal{J}^{*+}$- related.

Next we recall some results from \cite{maity2, maity1} that will be crucial in our study.

\begin{theorem}\label{17}
The following conditions on a semiring $(S,+,\cdot)$ are equivalent:

(i) $S$ is additively quasi regular with exactly one additive idempotent;

(ii) $S$ is a quasi skew-ring;

(iii) $S$ is a nil-extension of a skew-ring.
    
\end{theorem}

\begin{theorem}\label{2}
If $S$ is a quasi completely regular semiring then $Reg^{+}(S)=Cr(S)$.
\end{theorem}

\begin{theorem}\label{1}
The following conditions on a semiring
$(S, +, \cdot)$ are equivalent:

(i) $S$ is a quasi completely regular semiring;

(ii) every $\mathscr{H}^{*+}$- class is a quasi
skew-ring;

(iii) $S$ is a (disjoint) union of quasi skew-rings;

(iv) $S$ is a b-lattice of completely Archimedean
semirings;

(v) $S$ is an idempotent semiring of quasi skew-rings.
\end{theorem}

Lastly, a \emph{partial semigroup} is a structure $(S,+)$, 
where $`+$' is a \emph{partial binary operation}, i.e., for all $a, b, c \in S$, $(a + b) + c = a + (b + c)$ in the sense that if either side is defined, then so is the other and they are equal.
Similarly, considering both operations $`+$' and $`\cdot$' as partial binary operations and defining analogously we can construct \emph{partial semirings}.

For other notations and terminologies not given in this paper,
the reader is referred to the texts of Bogdanovic \cite{bogda1},
Howie \cite{howie}, Golan \cite{golan}, and Hebisch and Weinert
\cite{hwh}.

\section{Quasi Completely Inverse Semirings}

In the recent works \cite{maity2, maity1, maity3} the characterizations of quasi completely regular semirings have been entrenched. The aim of the study was to extend the concepts of decomposition of a semigroup into semilattice of subsemigroups to the domain of semirings. In \cite{maity4} quasi completely inverse semirings have been illustrated where some more structural characterizations occurred due to the introduction of inverse criterion to quasi completely regular semirings. This section mainly retrospects some of the important structural characterizations of quasi completely inverse semirings in an enhanced way as required for our discussion.

\begin{definition} \cite{maity4}
A semiring $S$ is a quasi completely inverse semiring
if $S$ is quasi completely regular semiring and additively quasi
inverse semiring.
\end{definition}

\begin{theorem}
\label{3} \cite{maity4}
The following conditions are equivalent on a semiring $(S, +, \cdot)$:

(i) $S$ is a quasi completely inverse semiring;

(ii) $S$ is a quasi completely regular semiring and for every $a \in S$ and $f \in E^{+}(S)$ there exists $m \in \mathbb{N}$ such that
$m(a+f)$ = $m(f+a)$.

(iii) $S$ is a quasi completely regular semiring and for every $e,\,f \in E^{+}(S)$ there exists $m \in \mathbb{N}$ such that
$m(e+f)$ = $m(f+e)$.

(iv) $S$ is quasi completely regular and $(a+b)
\,\,\mathscr{H}^{*+} \,\,(b+a)$ for all $a, b \in S$;

(v) $S$ is b-lattice of quasi skew-rings.
\end{theorem}


\begin{theorem}\label{5}
A semiring $S$ is a quasi completely inverse semiring if and only if it is additively quasi regular and $\mathscr{H}^{*+}=\mathscr{J}^{*+}$. 
\end{theorem}

\begin{proof}
Let $S$ is a quasi completely inverse semiring. Then by Theorem \ref{3}, $S$ is b-lattice of quasi skew-rings which are the $\mathcal{H}^{*+}$-classes of $S$. Also from \cite[Corollary 3.21.]{maity2} it follows that each $\mathcal{H}^{*+}$-class is a completely Archimedean semiring with exactly one additive idempotent, i.e. a $\mathcal{J}^{*+}$-class with exactly one additive idempotent and conversely. 
\end{proof}

\begin{theorem}\label{6} 
Let $S$ be a quasi completely inverse semiring. Then $\mathscr{H}^{*+}$ is the greatest additive idempotent separating congruence on $S$. 
\end{theorem}

\begin{proof}
Follows from \cite[Theorem 4.10.]{maity4}. 
\end{proof}

\section{Strongly Additively Quasi Completely Inverse Semirings}

There are two explicit definitions of inverses on semigroups, $\pi$-inverse semigroups and strongly  $\pi$-inverse semigroups. In \cite{bogda1} it has been illustrated that a semigroup is strongly  $\pi$-inverse if and only if it is $\pi$-inverse and the product of any two idempotents of the semigroup is an idempotent. In this section we derive analogous results pertaining to the above ideas.  

\begin{definition}\cite{maity4} \label{7}
A semiring $S$ is strongly additively quasi inverse semiring if $S$ is additively
quasi regular and the additive idempotent elements in $S$ commute, i.e. for any two elements $e,f \in E^{+}(S)$, $e+f=f+e$. 
\end{definition}

\begin{definition} \cite{maity4} \label{8}
A semiring $S$ is strongly additively quasi completely inverse semiring if $S$ is quasi completely regular and strongly additively quasi inverse semiring. 
\end{definition}

\begin{theorem}\label{9}
The following conditions on a semiring
$(S, +, \cdot)$ are equivalent:

(i) $S$ is a strongly additively quasi completely inverse semiring,

(ii) $S$ is quasi completely regular and $Reg^{+}(S)$ is an additive inverse subsemiring of $S$.

(iii) $S$ is quasi completely inverse and for any two elements $e,f \in E^{+}(S)$, $e+f\in E^{+}(S)$.
\end{theorem}

\begin{proof}

$(i)\Longrightarrow (ii):$ Follows from \cite[Theorem 3.9]{maity4}.

$(ii)\Longrightarrow (i):$ Follows from \cite[Theorem 3.9]{maity4}.

$(i)\Longrightarrow (iii):$ Let $S$ be a strongly additively quasi completely inverse semiring. Then $S$ is quasi completely regular and strongly additively quasi inverse semiring, which further implies that $S$ is quasi completely regular and for any two elements $e,f \in E^{+}(S)$, $e+f=f+e$. Thus, from Theorem \ref{3} it follows that $S$ is quasi completely inverse semiring. 
Clearly, $e+f\in E^{+}(S)$ for any $e,f \in E^{+}(S)$.

$(iii)\Longrightarrow (i):$ $S$ is quasi completely inverse and for any two elements $e,f \in E^{+}(S)$, $e+f\in E^{+}(S)$. Clearly, $ef \in E^{+}(S)$. This implies $E^{+}(S)$ is a additive inverse subsemiring of $S$. Thus, from \cite[Theorem V.1.2]{howie}, it follows that the additive idempotent elements of $S$ commute. Thus $S$ is a strongly additively quasi completely inverse semiring.
\end{proof}

\begin{cor}\label{10}
  A semiring $S$ is a strongly additively quasi completely inverse semiring if and only if $S$ is a quasi completely inverse semiring with $Reg^{+}(S)$ and $E^{+}(S)$ being ideals of $S$.   
\end{cor}

\begin{definition} \label{11}
Let $S$ be a strongly additively quasi completely inverse semiring. Then $S$ is quasi completely inverse semiring where $e,f \in E^{+}(S)$ implies $e+f \in E^{+}(S)$. Also $S$ is a b-lattice of quasi skew-rings.
Let us denote $S$ by $S=\bigcup_{\alpha \in Y} T_{\alpha}$, where $Y$ is a b-lattice of quasi skew-rings $T_{\alpha}$ for each $\alpha \in Y$. Let $T_{\alpha}=R_{\alpha} \cup S_{\alpha}$, where $R_{\alpha}$ is the skew-ring kernel of $T_{\alpha}$, and the identity of $R_{\alpha}$ is denoted by $e_{\alpha}$ for any $\alpha \in Y$. $S_{\alpha}$ = $T_{\alpha}\setminus R_{\alpha}$ is the set of non-additively regular elements of $T_{\alpha}$ and $S_{\alpha}$ is a partial semiring by the definition of quasi skew-ring.

$(i)$ Define a mapping $\psi : S \longrightarrow Reg^{+}(S)$ by $\psi(a)=a+e_{\alpha}$, for any $a \in S$, where $e_{\alpha}$ is the unique additive idempotent of $T_{\alpha}$. Then $\psi  {\big{|} _{R_{\alpha}}}$ = $I_{R_{\alpha}}$, $I_{R_{\alpha}}$ being the identity mapping on $R_{\alpha}$.

$(ii)$ Define a relation $\widetilde{\psi}$ for any $a,b \in S$ as follows:

\begin{center}
  $a \,\,\widetilde{\psi}\,\,b$ if and only if $\psi(a)=\psi(b)$.    
\end{center}
   
Clearly, $\widetilde{\psi}{\big{|}_{Reg^{+}(S)}}=\epsilon$. Also, $a \,\,\widetilde{\psi}\,\,(a+e_{\alpha})$ for any $a \in T_{\alpha}$. $\widetilde{\psi}{\big{|}_{S}}=\epsilon$ if and only if $S$ is completely regular semiring where every additive idempotent element commute. In general, $\widetilde{\psi}{\big{|}_{S \setminus Reg^{+}(S)}} \neq \epsilon$. The following example justifies our claim.

\end{definition}

\begin{example}
Let $S=\{0,a,b\}$ be a semiring with the following Cayley tables:

\begin{table}[ht] \begin{center}
\begin{tabular}{c|ccc}
+   & $a$ & $b$ & $0$ \\ \hline
$a$ & $b$ & $0$ & $0$ \\
$b$ & $0$ & $0$ & $0$ \\
$0$ & $0$ & $0$ & $0$ \\
\end{tabular}
\qquad \qquad
\begin{tabular}{c|ccc}
$\cdot$ & $a$ & $b$ & $0$ \\ \hline
    $a$ & $b$ & $0$ & $0$ \\
    $b$ & $0$ & $0$ & $0$ \\
    $0$ & $0$ & $0$ & $0$ \\
\end{tabular}
\end{center}
\end{table}

Clearly, $S$ is a quasi skew-ring and $Reg^{+}(S)=\{0\}$ and $\psi(a)=a+0=0=b+0=\psi(b)$, but $a \neq b$.
\end{example}

\begin{remark}\label{rem1}
Let $S=\bigcup_{\alpha \in Y} T_{\alpha}$, where $Y$ is a b-lattice of quasi skew-rings $T_{\alpha}$ for each $\alpha \in Y$, be a strongly additively quasi completely inverse semiring. Let $a \in T_{\alpha}$, $b \in T_{\beta}$. Then, $a \, \mathscr{H^{*+}} \, \alpha$ and $b \, \mathscr{H^{*+}} \, \beta$ which implies $(a+b) \, \mathscr{H^{*+}} \, (\alpha+\beta)$ and $ab \, \mathscr{H^{*+}} \, \alpha \beta$. This further implies that $(a+b) \in T_{\alpha + \beta}$, $ab \in T_{\alpha \beta}$.    \end{remark}

\begin{theorem}\label{12}
Let $S$ be a strongly additively quasi completely inverse semiring. Then $\psi$ is a homomorphism. 
\end{theorem}
  
\begin{proof}
Let $a,b \in S$ and $a \in T_{\alpha}$, $b \in T_{\beta}$ for $\alpha, \beta \in Y$. Then $\psi(a)=a+e_{\alpha}$, and $\psi(b)=b+e_{\beta}$. $\psi(a+b)=a+b+e_{\alpha+\beta}$, $\psi(a)+\psi(b)=(a+e_{\alpha})+(b+e_{\beta})=a+e_{\alpha}+e_{\beta}+b=a+e_{\alpha+\beta}+b=a+b+e_{\alpha+\beta}$, which implies $\psi(a+b)=\psi(a)+\psi(b)$. Again, $\psi(ab)=ab+e_{\alpha \beta}$, $\psi(a) \psi(b)= (a+e_{\alpha})(b+e_{\beta})=ab+ae_{\beta}+e_{\alpha}b+e_{\alpha}e_{\beta}=ab+e_{\alpha \beta}$, which implies $\psi(ab)= \psi(a) \psi(b)$. This proves that $\psi$ is a homomorphism.
\end{proof}

\section{Strong b-lattice of quasi skew-rings}
In \cite{zhangmain} the authors derived the necessary and sufficient conditions for a GV-inverse semigroup to be a strong semillatice of $\pi$-groups. In \cite{sen1} the authors defined strong b-lattice of semirings analogous to the structure of strong semilattice of semigroups where they used b-lattice instead of semilattice. Using this definition they also established that a semirig is a generalized Clifford semiring if and only if it is a strong b-lattice of skew-rings. In this section we use the definition of strong b-lattice of semirings to prove our main result.
Before stating the main result of this paper, let us recall some established important definitions and results from \cite{sen1}.

\begin{definition} \label{14} \cite{sen1}
A semiring $(S,+,\cdot)$ is called a generalized Clifford semiring if it is a completely regular semiring where $(S,+)$ is inverse semigroup  and $E^{+}(S)$ is a k-ideal of $S$.
\end{definition}

\begin{definition}\label{15} \cite{sen1}
Let $T$ be  b-lattice and $\{S_{\alpha} : \alpha \in T \}$  be a family of pairwise disjoint semirings which are indexed by the elements of $T$. For each $\alpha \leq \beta $ in $T$, we now embed $ S_{\alpha} $ in  $S_{\beta}$ via a semiring monomorphism $\phi_{_{\alpha,\beta}}$ satisfying the
following conditions

\noindent
(1)   
\begin{center}
$\phi_{_{\alpha,\alpha}} = I_{S_{\alpha}}$, the identity mapping on $  S_{\alpha}$,
\end{center}

\noindent
(2)   
\vspace{-3em}
\begin{center}
$\phi_{_{\alpha,\beta}} \phi_{_{\beta,\gamma}} = \phi_{_{\alpha,\gamma}}$  \hspace{3em} if  $\alpha\leq\beta\leq\gamma$,
\end{center}
\noindent
(3) 
\vspace{-2.7em}
\begin{center}
$S_{\alpha} \phi_{_{\alpha,\gamma}}S_{\beta} \phi_{_{\beta,\gamma}}  \subseteq S_{\alpha\beta} \phi_{_{\alpha\beta,\gamma}}$ \hspace{.5em} if $ \alpha+\beta \leq
\gamma$, i.e., $ \alpha+\beta + \alpha\beta \leq
\gamma$ .
\end{center}

On  $ S = \displaystyle{\bigcup_{\alpha \in T}S_\alpha}$ we define  addition `\,$+$' and multiplication `\,$\cdot$' for $ a \in S_\alpha,$ $b \in S_\beta $, as  follows:

\noindent
(4) 
\vspace{-3em}
\begin{center}
$ a + b = a \phi_{_{\alpha,\alpha+\beta}}+b  \phi_{_{\beta,\alpha+\beta}}$ 
\end{center}

\noindent
and

\noindent
(5) 
\vspace{-2.6em}
\begin{center}
$a\cdot b = c \in S_{\alpha\beta} $ such  that $ c\phi_{_{\alpha\beta,\alpha+\beta}} = a\phi_{_{\alpha,\alpha+\beta}}\cdot b\phi_{_{\beta,\alpha+\beta}}$
\end{center}

\noindent We denote the above system by $ S = < T, S_\alpha, \phi_{_{\alpha,\beta}}> $ and call it the 
strong b-lattice \index{strong b-lattice of semirings} $T$ of the semirings $ S_\alpha, \alpha \in T$. Also, the system $ S = < T, S_\alpha, \phi_{_{\alpha,\beta}}> $ is a semiring.
\end{definition}

\begin{theorem} \label{16} \cite{sen1}
A semiring $S$ is a generalized Clifford semiring if and only if $S$ is a strong b-lattice of skew-rings.
\end{theorem}

\noindent
 Now we state the main result of this paper.

\begin{theorem}\label{13}
Let $S = \displaystyle{\bigcup_{\alpha \in Y}T_\alpha}$ be a strongly additively quasi completely inverse semiring. If the following conditions are satisfied:

    (i) $Reg^{+}(S)$ = $\displaystyle{\bigcup_{\alpha \in Y}R_\alpha}$ is a generalized Clifford subsemiring of $S$, denoted by $R<Y, R_\alpha, \theta_{_{\alpha,\beta}}> $, and it is a bi-ideal of $S$, i.e., $S$ is a nil-extension of generalized Clifford semiring.
    
    (ii) For any $\alpha, \beta \in Y$ with $\alpha \leq \beta$, if $S_{\alpha} \neq \phi$, there is a semiring monomorphism $\varphi_{_{\alpha,\beta}} : S_{\alpha} \longrightarrow T_{\beta}$ such that
    
    (1) $\varphi_{_{\alpha,\alpha}} = I_{S_{\alpha}}$, the identity mapping on $  S_{\alpha}$,

    (2) For any $\alpha, \beta, \gamma \in Y$ with $\alpha + \beta \leq \gamma$, if $a \in S_{\alpha}, b \in S_{\beta}$, $a+b \notin S_{\alpha+\beta}$ then $a\varphi_{_{\alpha,\gamma}} + b\varphi_{_{\beta,\gamma}} \notin S_{\gamma} $.

   (3) For any $\alpha, \beta, \gamma \in Y$ with $\alpha\leq\beta\leq\gamma$ and $a \in S_{\alpha}$, if $a\varphi_{_{\alpha,\beta}} \in S_{\beta}$, then $a\varphi_{_{\alpha,\beta}} \varphi_{_{\beta,\gamma}} = a\varphi_{_{\alpha,\gamma}}$. If $a\varphi_{_{\alpha,\beta}} \notin S_{\beta}$, then $a\varphi_{_{\alpha,\gamma}} \notin S_{\gamma}$,

    (4) For any $\alpha, \beta, \gamma \in Y$ with $\alpha + \beta \leq \gamma$ i.e. $\alpha + \beta+\alpha\beta \leq \gamma$,and $a \in S_{\alpha}, b \in S_{\beta}$, if $ab\in S_{\alpha\beta}$, then $a\varphi_{_{\alpha,\gamma}} b\varphi_{_{\beta,\gamma}}$ = $(ab)\varphi_{_{\alpha \beta,\gamma}}$ and if $ab \notin S_{\alpha \beta}$ then, $(a\varphi_{_{\alpha,\gamma}}) (b\varphi_{_{\beta,\gamma}}) \notin S_{\gamma} $,

    (5) For any $\alpha, \beta \in Y$ and $a \in S_{\alpha}, b \in S_{\beta}$ if $a+b\in S_{\alpha+\beta}$, $ab \in S_{\alpha \beta}$, then $a+b=a\varphi_{_{\alpha,\alpha+\beta}}+b\varphi_{_{\beta, \alpha+\beta}}$ and $(ab) \varphi_{_{\alpha\beta, \alpha+\beta}}=(a\varphi_{_{\alpha,\alpha+\beta}})(b\varphi_{_{\beta, \alpha+\beta}})$,

  (iii) For any $\alpha, \beta \in Y$ with $\alpha \leq \beta$, $\varphi_{_{\alpha, \beta}} \psi$ = $\psi \theta_{_{\alpha, \beta}}$, where $\psi$ is a homomorphism defined earlier.

Define a mapping $\Phi_{_{\alpha, \beta}} : T_{\alpha} \longrightarrow T_{\beta}$ for any $\alpha, \beta \in Y$ with $\alpha \leq \beta$ and $a \in T_{\alpha}$ as,

\begin{center}
 $a \Phi_{_{\alpha, \beta}}$ = $\Bigg\{$ $\begin{tabular}{cc}
  $ a \theta_{_{\alpha, \beta}}$, & $ a \in R_{\alpha}$ \\
$ a \varphi_{_{\alpha, \beta}}$, & $ a \in S_{\alpha}$ \\

 \end{tabular}$
\end{center}

Then $S$ is a strong b-lattice of quasi skew-rings, denoted by $S<Y, T_\alpha, \Phi_{_{\alpha,\beta}}> $ and conversely.
\end{theorem}

\begin{proof}
Let $S = \displaystyle{\bigcup_{\alpha \in Y}T_\alpha}$ be a strongly additively quasi completely inverse semiring and let the given conditions are satisfied.

\noindent
To prove $S$ to be a strong b-lattice of quasi skew-rings, we firstly show that the mapping $\Phi_{_{\alpha, \beta}}$ is a monomorphism from $T_{\alpha}$ to $T_{\beta}$. 

\noindent
Let $a,b \in T_{\alpha}$ and $\alpha, \beta \in Y$ with $\alpha \leq \beta$. Since $T_{\alpha}=R_{\alpha} \cup S_{\alpha}$, where $R_{\alpha}$ is the skew-ring kernel and $S_{\alpha}$= $T_{\alpha}\setminus R_{\alpha}$, there arises several cases. 

\noindent
Case 1 : if $a \in S_{\alpha}$, $b \in S_{\alpha}$, $a+b, ab \in S_{\alpha}$, since $\varphi_{_{\alpha,\beta}}$ is a semiring homomorphism, 

\begin{center}
    $(a+b)\Phi_{_{\alpha,\beta}}$ = $(a+b)\varphi_{_{\alpha,\beta}}$ = $a\varphi_{_{\alpha,\beta}} + b\varphi_{_{\alpha, \beta}}$ = $a\Phi_{_{\alpha,\beta}} + b\Phi_{_{\alpha, \beta}}$.

    $(ab)\Phi_{_{\alpha,\beta}}$ = $(ab)\varphi_{_{\alpha,\beta}}$ = $(a\varphi_{_{\alpha,\beta}})(b\varphi_{_{\alpha, \beta}})$ = $(a\Phi_{_{\alpha,\beta}})(b\Phi_{_{\alpha, \beta}})$.
\end{center}

\noindent
Case 2 : if $a \in S_{\alpha}$, $b \in S_{\alpha}$, $a+b, ab \in R_{\alpha}$, the skew-ring kernel of $T_{\alpha}$, then

\begin{center}

\begin{tabular}{ccl}
    $(a+b)\Phi_{_{\alpha,\beta}}$\,\,  &  = & $(a+b+e_{\alpha})\Phi_{_{\alpha,\beta}}$ \\
     $ $ & = & $(a+b+e_{\alpha})\theta_{_{\alpha,\beta}}$ \\
     $ $ & = & $((a+e_{\alpha})+(b+e_{\alpha}))\theta_{_{\alpha,\beta}}$ \\
     $ $ & = & $(a+e_{\alpha})\theta_{_{\alpha,\beta}}+(b+e_{\alpha})\theta_{_{\alpha,\beta}}$ (by Condition (i)) \\
     $ $ & = & $a \psi\theta_{_{\alpha,\beta}} + b \psi\theta_{_{\alpha,\beta}}$ (by definition of $\psi$) \\
     $ $ & = & $a \varphi_{_{\alpha, \beta}} \psi$ + $b \varphi_{_{\alpha, \beta}} \psi$ (by Condition (iii))\\
     $ $ & = & $(a \varphi_{_{\alpha, \beta}} + e_{\beta})$ + $(b \varphi_{_{\alpha, \beta}} + e_{\beta})$ \\
     $ $ & = & $a \varphi_{_{\alpha, \beta}} + b \varphi_{_{\alpha, \beta}} + e_{\beta}$ \\
     $ $ & = & $a \varphi_{_{\alpha, \beta}} + b \varphi_{_{\alpha, \beta}}$ (by Condition (ii)(2))  \\
     $ $ & = & $a \Phi_{_{\alpha, \beta}} + b \Phi_{_{\alpha, \beta}}$.
\end{tabular}
  
\begin{tabular}{ccl}
    $(ab)\Phi_{_{\alpha,\beta}}$  &  = & $(ab+e_\alpha)\Phi_{_{\alpha,\beta}}$ \\
     $ $ & = & $(ab+e_{\alpha})\theta_{_{\alpha,\beta}}$ \\
     $ $ & = & $(ab) \psi\theta_{_{\alpha,\beta}}$ \\
    $ $ & = & $(a \psi b \psi)\theta_{_{\alpha,\beta}}$ (since $\psi$ is a homomorphism) \\
     $ $ & = & $(a \psi \theta_{_{\alpha,\beta}})(b \psi \theta_{_{\alpha,\beta}})$ ( by Condition (i)) \\
     $ $ & = & $(a \varphi_{_{\alpha, \beta}} \psi$)($b \varphi_{_{\alpha, \beta}} \psi$) (by Condition (iii)) \\
     $ $ & = & $(a \varphi_{_{\alpha, \beta}} + e_{\beta})$ $(b \varphi_{_{\alpha, \beta}} + e_{\beta})$ \\
     $ $ & = & $(a \varphi_{_{\alpha, \beta}})(b \varphi_{_{\alpha, \beta}}) + e_{\beta}$ \\
     $ $ & = & $(a \varphi_{_{\alpha, \beta}})(b \varphi_{_{\alpha, \beta}})$ (by Condition (ii)(4))  \\
     $ $ & = & $(a \Phi_{_{\alpha, \beta}})(b \Phi_{_{\alpha, \beta}})$.
\end{tabular}
    
\end{center}

\noindent
  Case 3 : if $a \in S_{\alpha}$, $b \in S_{\alpha}$, $a+b \in S_{\alpha}$ and $ab \in R_{\alpha}$, the skew-ring kernel of $T_{\alpha}$, then the proof follows from Case 1 and Case 2.

\noindent
  Case 4 : if $a \in S_{\alpha}$, $b \in S_{\alpha}$, $a+b \in R_{\alpha}$, the skew-ring kernel of $T_{\alpha}$ and $ab \in S_{\alpha}$,  then the proof follows from Case 1 and Case 2.

  \noindent
  Case 5 : if $a \in S_{\alpha}$, $b \in R_{\alpha}$ then $a+b, ab\in R_{\alpha}$, the skew-ring kernel of $T_{\alpha}$. Therefore,

\begin{center}

\begin{tabular}{ccl}
    $(a+b)\Phi_{_{\alpha,\beta}}$  &  = & $(a+b+e_{\alpha})\Phi_{_{\alpha,\beta}}$ \\
     $ $ & = & $(a+b+e_{\alpha})\theta_{_{\alpha,\beta}}$ \\
     $ $ & = & $(a+(e_{\alpha}+b))\theta_{_{\alpha,\beta}}$ \\
     $ $ & = & $(a+e_{\alpha})\theta_{_{\alpha,\beta}}+b\theta_{_{\alpha,\beta}}$  ( by Condition (i)) \\
     $ $ & = & $a \psi\theta_{_{\alpha,\beta}} + b \theta_{_{\alpha,\beta}}$ (by definition of $\psi$) \\
     $ $ & = & $a \varphi_{_{\alpha, \beta}} \psi$ + $b \theta_{_{\alpha,\beta}}$ (by Condition (iii))\\
     $ $ & = & $(a \varphi_{_{\alpha, \beta}} + e_{\beta})$ + $b \theta_{_{\alpha,\beta}}$ \\
     $ $ & = & $a \varphi_{_{\alpha, \beta}} + (e_{\beta}$ + $b \theta_{_{\alpha,\beta}})$ \\
     $ $ & = & $a \varphi_{_{\alpha, \beta}}$ + $b \theta_{_{\alpha,\beta}}$ \\
     $ $ & = & $a \Phi_{_{\alpha, \beta}}$ + $b \Phi_{_{\alpha,\beta}}$.
     
\end{tabular}
  
\begin{tabular}{ccl}
    $\,\,\,\,\,\,\,\,\,\,\,\,\,(ab)\Phi_{_{\alpha,\beta}}\,\,\,\,\,$  &  = & $(ab+e_\alpha)\Phi_{_{\alpha,\beta}}$ \\
     $ $ & = & $(ab+e_{\alpha})\theta_{_{\alpha,\beta}}$ \\
     $ $ & = & $((a+e_{\alpha})b) \theta_{_{\alpha,\beta}}$ \\
    $ $ & = & $((a+e_{\alpha}) \theta_{_{\alpha,\beta}})(b \theta_{_{\alpha,\beta}}) $ (by Condition(i)) \\
     $ $ & = & $(a \psi \theta_{_{\alpha,\beta}})(b \theta_{_{\alpha,\beta}})$\\
     $ $ & = & $(a \varphi_{_{\alpha, \beta}} \psi$)($b \theta_{_{\alpha,\beta}} $)  (by Condition (iii)) \\
     $ $ & = & $(a \varphi_{_{\alpha, \beta}} + e_{\beta})$ $(b \theta_{_{\alpha,\beta}})$ \\
     $ $ & = & $(a \varphi_{_{\alpha, \beta}})(b \theta_{_{\alpha,\beta}}) + e_{\beta}$ \\
      $ $ & = & $(a \varphi_{_{\alpha, \beta}})(b \theta_{_{\alpha,\beta}})$ (since $R_{\beta}$ is a bi-ideal of $T_{\beta}$)\\
     $ $ & = & $(a \Phi_{_{\alpha, \beta}})(b \Phi_{_{\alpha, \beta}})$.
\end{tabular}
    
\end{center}

\noindent
  Case 6 : if $a \in R_{\alpha}$, $b \in S_{\alpha}$ then $a+b, ab\in R_{\alpha}$. Applying similar arguments as in Case 5, we have,

\begin{center}

\begin{tabular}{ccl}
    \,\,\,\,$(a+b)\Phi_{_{\alpha,\beta}}$ \,\,\,\,&  = & $(a+b+e_{\alpha})\Phi_{_{\alpha,\beta}}$ \\
     $ $ & = & $(a+b+e_{\alpha})\theta_{_{\alpha,\beta}}$ \\
     $ $ & = & $a\theta_{_{\alpha,\beta}}$ + $(b+e_{\alpha})\theta_{_{\alpha,\beta}}$ \\
     $ $ & = & $a\theta_{_{\alpha,\beta}}+b \psi \theta_{_{\alpha,\beta}}$ \\
      $ $ & = & $a\theta_{_{\alpha,\beta}}+b \varphi_{_{\alpha, \beta}} \psi$ (by Condition (iii))\\
       $ $ & = & $a\theta_{_{\alpha,\beta}} + b \varphi_{_{\alpha, \beta}} + e_{\beta}$ \\
     $ $ & = & $a\theta_{_{\alpha,\beta}} + b \varphi_{_{\alpha, \beta}}$ (since $R_{\beta}$ is a bi-ideal of $T_{\beta}$)\\
     $ $ & = & $a\Phi_{_{\alpha,\beta}} + b \Phi_{_{\alpha, \beta}} $.
\end{tabular}
  
\begin{tabular}{ccl}
    $\,\,\,\,\,\,\,\,\,\,\,\,\,\,\,\,\,\,(ab)\Phi_{_{\alpha,\beta}}\,\,\,\,\,\,\,\,\,\,\,\,$  &  = & $(ab+e_\alpha)\Phi_{_{\alpha,\beta}}$ \\
     $ $ & = & $(ab+e_{\alpha})\theta_{_{\alpha,\beta}}$ \\
     $ $ & = & $(a(b+e_{\alpha})) \theta_{_{\alpha,\beta}}$ \\
    $ $ & = & $(a \theta_{_{\alpha,\beta}})((b+e_{\alpha}) \theta_{\alpha,\beta}) $ (by Condition (i)) \\
      $ $ & = & $(a \theta_{_{\alpha,\beta}})(b\psi\theta_{_{\alpha,\beta}}) $ \\
     $ $ & = & $(a \theta_{_{\alpha,\beta}})(b\varphi_{_{\alpha, \beta}}\psi)$  (by Condition (iii)) \\
     $ $ & = & $(a \theta_{_{\alpha,\beta}})(b \varphi_{_{\alpha, \beta}} + e_{\beta})$ \\
     $ $ & = & $(a \theta_{_{\alpha,\beta}})(b \varphi_{_{\alpha, \beta}}) + e_{\beta}$ \\
      $ $ & = & $(a \theta_{_{\alpha,\beta}})(b \varphi_{_{\alpha, \beta}})$ (since $R_{\beta}$ is a bi-ideal of $T_{\beta}$) \\
     $ $ & = & $(a \Phi_{_{\alpha, \beta}})(b \Phi_{_{\alpha, \beta}})$.
\end{tabular}
    
\end{center}

\noindent
Case 7: if $a \in R_{\alpha}$, $b \in R_{\alpha}$ then $a+b, ab\in R_{\alpha}$. By Condition (i),

\begin{center}
    $(a+b)\Phi_{\alpha,\beta}$ = $(a+b)\theta_{\alpha,\beta}$ = $a \theta_{\alpha,\beta} + b\theta_{\alpha, \beta}$ = $a\Phi_{\alpha,\beta} + b\Phi_{\alpha, \beta}$.

    $(ab)\Phi_{\alpha,\beta}$ = $(ab)\theta_{\alpha,\beta}$ = $(a\theta_{\alpha,\beta})(b\theta_{\alpha, \beta})$ = $(a\Phi_{\alpha,\beta})(b\Phi_{\alpha, \beta})$.
\end{center}

\noindent
Thus we see that the mapping $\Phi_{_{\alpha, \beta}}$ is a homomorphism from $T_{\alpha}$ to $T_{\beta}$.

\noindent
Clearly, from Theorem \ref{16} and Condition (i), $\theta_{_{\alpha,\beta}}$ is injective. Also, from Condition (ii), $\varphi_{_{\alpha,\beta}}$ is injective. These imply that $\Phi_{_{\alpha, \beta}}$ is injective and thus a monomorphism.

\noindent
From Theorem \ref{16}, Condition (i), Condition (ii)(1) and definition of $\Phi_{_{\alpha, \beta}}$, 

\begin{center}
 $\Phi_{_{\alpha,\alpha}} = I_{_{T_{\alpha}}}$, the identity mapping on $  T_{\alpha}$.   
\end{center}

\noindent
Let $\alpha, \beta, \gamma \in Y $ with $\alpha \leq \beta \leq \gamma$. If $a \in R_{\alpha}$ then clearly by Condition (i), $a\Phi_{_{\alpha,\beta}} \Phi_{_{\beta,\gamma}}$ = $a\Phi_{_{\alpha,\gamma}}$. Now let $a \in S_{\alpha}$.

\noindent
If $a \varphi_{_{\alpha, \beta}} \in S_{\beta}$, then by Condition (ii)(3), 

\begin{center}
$a\Phi_{_{\alpha,\beta}} \Phi_{_{\beta,\gamma}}$ = $a\varphi_{_{\alpha,\beta}} \Phi_{_{\beta,\gamma}}$ = $a\varphi_{_{\alpha,\beta}} 
 \varphi_{_{\beta,\gamma}}$ = $a\varphi_{_{\alpha,\gamma}}$ = $a\Phi_{_{\alpha,\gamma}}$.    
\end{center}

\noindent
If $a \varphi_{_{\alpha, \beta}} \notin S_{\beta}$, then by Condition (ii)(3), $a\varphi_{_{\alpha,\gamma}} \notin S_{\gamma}$. Now,

\begin{center}
\begin{tabular}{ccl}
   $a\Phi_{_{\alpha,\beta}} \Phi_{_{\beta,\gamma}}$  &  = & $a\varphi_{_{\alpha,\beta}} \theta_{_{\beta,\gamma}}$ \\
   $ $     &  = & $(a\varphi_{_{\alpha,\beta}}+e_{\beta})\theta_{_{\beta,\gamma}}$ \\
   $ $     &  = & $a\varphi_{_{\alpha,\beta}}\psi\theta_{_{\beta,\gamma}}$ \\
   $ $     &  = & $a \psi \theta_{_{\alpha,\beta}} \theta_{_{\beta,\gamma}}$ (by Condition (iii))\\
   $ $     &  = & $a \psi \theta_{_{\alpha,\gamma}}$ (by Condition (i))\\
   $ $     &  = & $a \varphi_{_{\alpha,\gamma}} \psi$ (by Condition (iii))\\
   $ $    &   = & $a \varphi_{_{\alpha,\gamma}}+e_{\gamma}$ \\
   $ $    &   = & $a \varphi_{_{\alpha,\gamma}}$ (by Condition (ii)(3)) \\
   $ $    &   = & $a \Phi_{_{\alpha,\gamma}}$.
\end{tabular} 
\end{center}

\noindent
For $\alpha, \beta, \gamma \in Y$, with $ \alpha+\beta \leq
\gamma$. In $Y$, we always have $ \alpha+\beta = \alpha+\beta + \alpha\beta$, i.e. if  $ \alpha+\beta \leq
\gamma$, $\alpha+\beta + \alpha\beta \leq
\gamma$. Let $a \in T_{\alpha}$, $b \in T_{\beta}$ then by Remark \ref{rem1}, $(a+b) \in T_{\alpha + \beta}$, $ab \in T_{\alpha \beta}$.   Now we have the following cases:

\noindent
Case 1: if $a \in S_{\alpha}$, $b \in S_{\beta}$, $ab \in S_{\alpha \beta}$, we have by Condition (ii)(4), 
$a \Phi_{_{\alpha, \gamma}} b \Phi_{_{\beta, \gamma}}$ = $a \varphi_{_{\alpha, \gamma}} b \varphi_{_{\beta, \gamma}}$ = $ab \varphi_{_{\alpha \beta, \gamma}}$ = $ab \Phi_{_{\alpha \beta, \gamma}}$, i.e. $S_{\alpha} \Phi_{_{\alpha, \gamma}} S_{\beta} \Phi_{_{\beta, \gamma}} \subseteq S_{\alpha \beta} \Phi_{_{\alpha \beta, \gamma}}$.

\noindent
Case 2:  if $a \in S_{\alpha}$, $b \in S_{\beta}$, $ab \in R_{\alpha \beta}$, which means $ab \notin S_{\alpha \beta}$, then 

\begin{center}
    \begin{tabular}{ccl}
       $a \Phi_{_{\alpha, \gamma}} b \Phi_{_{\beta, \gamma}}$   & = & $a \varphi_{_{\alpha, \gamma}} b \varphi_{_{\beta, \gamma}}$ \\
       $ $  & = & $a \varphi_{_{\alpha, \gamma}} b \varphi_{_{\beta, \gamma}}$ +$e_{\gamma}$ (by Condition (ii)(4)) \\
       $ $  & = & $(a \varphi_{_{\alpha, \gamma}}+e_{\gamma})(b \varphi_{_{\beta, \gamma}}+e_{\gamma})$\\
       $ $  & = & $(a \varphi_{_{\alpha, \gamma}} \psi)(b \varphi_{_{\beta, \gamma}} \psi)$ \\
       $ $  & = & $(a \psi \theta_{_{\alpha, \gamma}})(b \psi \theta_{_{\beta, \gamma}})$ (by Condition (iii))\\
       $ $  & = & $(a+e_{\alpha})\theta_{_{\alpha, \gamma}} (b+e_{\beta}) \theta_{_{\beta, \gamma}}$ \\
       $ $  & = & $(a+e_{\alpha}+e_{\gamma})(b+e_{\beta}+e_{\gamma})$ \\
       $ $  & = & $(a+e_{\gamma})(b+e_{\gamma})$ \\
       $ $  & = & $ab+e_{\gamma}$ \\
       $ $  & = & $ab \theta_{_{\alpha \beta, \gamma}}$ \\
       $ $  & = & $ab \Phi_{_{\alpha \beta, \gamma}} $.
       
    \end{tabular}
\end{center}

Thus, $S_{\alpha} \Phi_{_{\alpha, \gamma}} S_{\beta} \Phi_{_{\beta, \gamma}} \subseteq R_{\alpha \beta} \Phi_{_{\alpha \beta, \gamma}}$.

\noindent
Case 3: if $a \in R_{\alpha}$, $b \in R_{\beta}$ then $ab \in R_{\alpha \beta}$ and $R_{\alpha} \Phi_{_{\alpha, \gamma}} R_{\beta} \Phi_{_{\beta, \gamma}} \subseteq R_{\alpha \beta} \Phi_{_{\alpha \beta, \gamma}}$ (by Condition (i)).

\noindent
Case 4: if $a \in R_{\alpha}$, $b \in S_{\beta}$ then $ab \in R_{\alpha \beta}$ and we have,  

\begin{center}
    \begin{tabular}{ccl}
       \,\,\,\,\,\,$a \Phi_{_{\alpha, \gamma}} b \Phi_{_{\beta, \gamma}}$\,\,\,\,\,\,  &  = & $a \theta_{_{\alpha, \gamma}} b \varphi_{_{\beta, \gamma}}$ \\
       $ $  & = & $a\theta_{_{\alpha, \gamma}} b \varphi_{_{\beta, \gamma}}$ + $e_\gamma$ (Since $R_{\gamma}$ is a bi-ideal of $T_\gamma$)\\
        $ $ & = & $a\theta_{_{\alpha, \gamma}}(b\varphi_{_{\beta, \gamma}}$ + $e_\gamma)$\\
        $ $ & = & $a\theta_{_{\alpha, \gamma}}b\varphi_{_{\beta, \gamma}} \psi$\\
        $ $ & = & $a\theta_{_{\alpha, \gamma}} b\psi \theta_{_{\beta, \gamma}}$\\
        $ $ & = & $a\theta_{_{\alpha, \gamma}} (b+e_{\beta}) \theta_{_{\beta, \gamma}}$\\
        $ $ & = & $(a+e_\gamma)(b+e_{\beta}+e_\gamma)$\\
        $ $ & = & $(a+e_\gamma)(b+e_\gamma)$\\
        $ $ & = & $ab+e_\gamma$\\
        $ $  & = & $ab \theta_{_{\alpha \beta, \gamma}}$ \\
       $ $  & = & $ab \Phi_{_{\alpha \beta, \gamma}} $.
        
    \end{tabular}
\end{center}

 This proves $R_{\alpha} \Phi_{_{\alpha, \gamma}} S_{\beta} \Phi_{_{\beta, \gamma}} \subseteq R_{\alpha \beta} \Phi_{_{\alpha \beta, \gamma}}$.

\noindent
Case 5: if $a \in S_{\alpha}$, $b \in R_{\beta}$ then $ab \in R_{\alpha \beta}$ and we have,  

\begin{center}
    \begin{tabular}{ccl}
       $\,\,\,\,\,\,\,\,a \Phi_{_{\alpha, \gamma}} b \Phi_{_{\beta, \gamma}}$ \,\,\,\,\,&  = & $a \varphi_{_{\alpha, \gamma}} b \theta_{_{\beta, \gamma}}$ \\
       $ $  & = &  $a \varphi_{_{\alpha, \gamma}} b \theta_{_{\beta, \gamma}} +e_\gamma$  (Since $R_\gamma$ is a bi-ideal of $T_\gamma$)\\
        $ $ & = & $(a\varphi_{_{\alpha, \gamma}}$ + $e_\gamma)b\theta_{_{\beta, \gamma}}$\\
        $ $ & = & $a\varphi_{_{\alpha, \gamma}} \psi b\theta_{_{\beta, \gamma}}$\\
        $ $ & = & $a \psi\theta_{_{\alpha, \gamma}} b\theta_{_{\beta, \gamma}}$\\
        $ $ & = & $(a+ e_\alpha)\theta_{_{\alpha, \gamma}} b\theta_{_{\beta, \gamma}}$\\
        $ $ & = & $(a + e_\alpha + e_\gamma) (b+ e_\gamma)$\\
        $ $ & = & $(a + e_\gamma) (b+ e_\gamma)$\\
        $ $ & = & $ab+e_\gamma$\\
        $ $  & = & $ab \theta_{_{\alpha \beta, \gamma}}$ \\
       $ $  & = & $ab \Phi_{_{\alpha \beta, \gamma}} $.
        
    \end{tabular}
\end{center}
 This proves $S_{\alpha} \Phi_{_{\alpha, \gamma}} R_{\beta} \Phi_{_{\beta, \gamma}} \subseteq R_{\alpha \beta} \Phi_{_{\alpha \beta, \gamma}}$.
 
\noindent
Thus we conclude that $T_{\alpha} \Phi_{_{\alpha, \gamma}} T_{\beta} \Phi_{_{\beta, \gamma}} \subseteq T_{\alpha \beta} \Phi_{_{\alpha \beta, \gamma}}$ if $ \alpha+\beta \leq
\gamma$.

\noindent
Next we consider addition and multiplication on $S$. For any $\alpha, \beta \in Y$, let $ a \in T_\alpha,$ and $b \in T_\beta $. The following are the cases that may arise :

\noindent
Case 1: if $a \in S_{\alpha}$, $b \in S_{\beta}$, $a+b \in S_{\alpha + \beta}$ and $ab \in S_{\alpha \beta}$, then by Condition (ii)(5) we have $a+b=a\varphi_{_{\alpha,\alpha+\beta}}+b\varphi_{_{\beta, \alpha+\beta}}$ = $a\Phi_{_{\alpha,\alpha+\beta}}+b\Phi_{_{\beta, \alpha+\beta}}$ and $(ab) \Phi_{_{\alpha\beta, \alpha+\beta}}$ = $(ab) \varphi_{_{\alpha\beta, \alpha+\beta}}=(a\varphi_{_{\alpha,\alpha+\beta}})(b\varphi_{_{\beta, \alpha+\beta}})$ = $(a\Phi_{_{\alpha,\alpha+\beta}})(b\Phi_{_{\beta, \alpha+\beta}})$.

\noindent
Case 2: if $a \in S_{\alpha}$, $b \in S_{\beta}$, $a+b \in R_{\alpha + \beta}$ and $ab \in R_{\alpha \beta}$, then 

\begin{center}
    \begin{tabular}{ccl}
        \,\,\,\,\,\,\,\,\,\,\,\,$\,\,\,\,\,\,\,\,\,\,a+b\,\,\,\,\,\,\,\,\,\,$ & = & $(a+b)+e_{\alpha + \beta}$ \,\,\,\,\,\,\,\,\,\,\\
        $ $ & = & $a+(b+e_{\alpha + \beta})$ \\
        $ $   & = & $a+e_{\alpha + \beta}+(b+e_{\alpha + \beta})$ (Since $R_{\alpha+\beta}$ is a bi-ideal of $T_{\alpha+\beta}$)\\
        $ $   & = & $(a+e_\alpha)+e_{\alpha + \beta}+(b+e_{\beta})+e_{\alpha + \beta}$ \\
        $ $   & = & $(a+e_\alpha) \theta_{_{\alpha, \alpha + \beta}}+(b+e_{\beta})\theta_{_{\beta, \alpha + \beta}}$ 
 (by Condition (i))\\
        $ $   & = & $a \psi \theta_{_{\alpha, \alpha + \beta}} + b \psi \theta_{_{\beta, \alpha + \beta}} $ \\
        $ $   & = & $a \varphi_{_{\alpha, \alpha + \beta}} \psi + b \varphi_{_{\beta, \alpha + \beta}} \psi$ (by Condition (iii)) \\
        $ $   & = & $(a \varphi_{_{\alpha, \alpha + \beta}}+e_{\alpha+\beta}) + (b \varphi_{_{\beta, \alpha + \beta}}+ e_{\alpha + \beta})$ \\
        $ $   & = & $a \varphi_{_{\alpha, \alpha + \beta}}+(e_{\alpha+\beta} + b \varphi_{_{\beta, \alpha + \beta}}+ e_{\alpha + \beta})$ \\
        $ $   & = & $a \varphi_{_{\alpha, \alpha + \beta}} + b \varphi_{_{\beta, \alpha + \beta}} + e_{\alpha + \beta}$ \\
        $ $   & = & $a \varphi_{_{\alpha, \alpha + \beta}} + b \varphi_{_{\beta, \alpha + \beta}}$ (by Condition (ii)(2)) \\
        $ $   & = & $a \Phi_{_{\alpha, \alpha + \beta}}+b \Phi_{_{\beta, \alpha + \beta}}$.
    \end{tabular}
\end{center}

\begin{center}
    \begin{tabular}{ccl}
        $(ab) \Phi_{_{\alpha \beta, \alpha+\beta}}$ & = & $(ab) \theta_{_{\alpha \beta, \alpha+\beta}}$ \\
        $ $ & = & $ab+e_{\alpha + \beta}$ \\
        $ $   & = & $(a+e_{\alpha + \beta})(b+e_{\alpha + \beta})$ \\
        $ $   & = & $((a+e_\alpha)+e_{\alpha + \beta})((b+e_{\beta})+e_{\alpha + \beta})$ \\
        $ $   & = & $((a+e_\alpha) \theta_{_{\alpha, \alpha + \beta}})((b+e_{\beta})\theta_{_{\beta, \alpha + \beta}})$ 
 (by Condition (i))\\
        $ $   & = & $(a \psi \theta_{_{\alpha, \alpha + \beta}})(b \psi \theta_{_{\beta, \alpha + \beta}}) $ \\
        $ $   & = & $(a \varphi_{_{\alpha, \alpha + \beta}} \psi)(b \varphi_{_{\beta, \alpha + \beta}} \psi$) (by Condition (iii)) \\
        $ $   & = & $(a \varphi_{_{\alpha, \alpha + \beta}}+e_{\alpha+\beta}) (b \varphi_{_{\beta, \alpha + \beta}}+ e_{\alpha + \beta})$ \\
        $ $   & = & $(a \varphi_{_{\alpha, \alpha + \beta}})(b \varphi_{_{\beta, \alpha + \beta}})+ e_{\alpha + \beta}$ \\
        $ $   & = & $(a \varphi_{_{\alpha, \alpha + \beta}})(b \varphi_{_{\beta, \alpha + \beta}})$ (by Condition (ii)(4)) \\
        $ $   & = & $(a \Phi_{_{\alpha, \alpha + \beta}})(b \Phi_{_{\beta, \alpha + \beta}})$.
    \end{tabular}
\end{center}

\noindent
Case 3: if $a \in S_{\alpha}$, $b \in S_{\beta}$, $a+b \in R_{\alpha + \beta}$ and $ab \in S_{\alpha \beta}$, then the results can be derived from Case 1 and Case 2.

\noindent
Case 4: if $a \in S_{\alpha}$, $b \in S_{\beta}$, $a+b \in S_{\alpha + \beta}$ and $ab \in R_{\alpha \beta}$, then the results can be derived from Case 1 and Case 2.

\noindent
Case 5: if $a \in R_{\alpha}$, $b \in R_{\beta}$, then $a+b \in R_{\alpha + \beta}$ and $ab \in R_{\alpha \beta}$, and $a+b=a\theta_{_{\alpha,\alpha+\beta}}+b\theta_{_{\beta, \alpha+\beta}}$ = $a\Phi_{_{\alpha,\alpha+\beta}}+b\Phi_{_{\beta, \alpha+\beta}}$ and $(ab) \Phi_{_{\alpha\beta, \alpha+\beta}}$ = $(ab) \theta_{_{\alpha\beta, \alpha+\beta}}=(a\theta_{_{\alpha,\alpha+\beta}})(b\theta_{_{\beta, \alpha+\beta}})$ = $(a\Phi_{_{\alpha,\alpha+\beta}})(b\Phi_{_{\beta, \alpha+\beta}})$ (by Condition (i)).

\noindent
Case 6: if $a \in R_{\alpha}$, $b \in S_{\beta}$ then since $ Reg^{+}(S)$ is a bi-ideal we have $a+b \in R_{\alpha + \beta}$ and $ab \in R_{\alpha \beta}$. Then,

\begin{center}
    \begin{tabular}{ccl}
        $\,\,\,a+b \,\,\,$ & = & $(a+b)+e_{\alpha + \beta}$\\
        $ $ & = & $a+e_{\alpha + \beta}+(b+e_{\alpha+\beta})$ \\
        $ $   & = & $a \theta_{_{\alpha, \alpha + \beta}}+(b+e_{\alpha + \beta})$ 
 (by Condition (i))\\
        $ $   & = & $a \theta_{_{\alpha, \alpha + \beta}}+(b+e_{\beta})+e_{\alpha + \beta}$ \\
        $ $   & = & $a \theta_{_{\alpha, \alpha + \beta}}+b \psi\theta_{_{\beta,\alpha + \beta}}$ (by Condition (i))\\
        $ $   & = & $a \theta_{_{\alpha, \alpha + \beta}}+b \varphi_{_{\beta,\alpha + \beta}} \psi$ (by Condition (iii))\\
        $ $   & = & $a \theta_{_{\alpha, \alpha + \beta}}+b \varphi_{_{\beta,\alpha + \beta}}+e_{\alpha+\beta}$ \\
        $ $   & = & $(a \theta_{_{\alpha, \alpha + \beta}}+b \varphi_{_{\beta,\alpha + \beta}})+e_{\alpha+\beta}$ \\
        $ $   & = & $(a \theta_{_{\alpha, \alpha + \beta}}+b \varphi_{_{\beta,\alpha + \beta}})$\\
        $ $   & = & $a \Phi_{_{\alpha, \alpha + \beta}}+b \Phi_{_{\beta, \alpha + \beta}}$.
    \end{tabular}
\end{center}

\begin{center}
    \begin{tabular}{ccl}
        $(ab) \Phi_{_{\alpha \beta, \alpha+\beta}}$ & = & $(ab) \theta_{_{\alpha \beta, \alpha+\beta}}$ \\
        $ $ & = & $ab+e_{\alpha + \beta}$ \\
        $ $   & = & $(a+e_{\alpha + \beta})(b+e_{\alpha + \beta})$ \\
        $ $   & = & $(a \theta_{_{\alpha, {\alpha + \beta}}})(b+e_{\alpha + \beta})$ \\
        $ $   & = & $(a \theta_{_{\alpha, {\alpha + \beta}}})((b+e_{\beta})+e_{\alpha + \beta})$ \\
        $ $   & = &  $(a \theta_{_{\alpha, {\alpha + \beta}}})((b+e_{\beta})\theta_{_{\beta, \alpha + \beta}})$ 
 (by Condition (i))\\
        $ $   & = & $(a \theta_{_{\alpha, {\alpha + \beta}}})(b \psi \theta_{_{\beta, \alpha + \beta}}) $ \\
        $ $   & = & $(a \theta_{_{\alpha, {\alpha + \beta}}})(b \varphi_{_{\beta, \alpha + \beta}} \psi$) (by Condition (iii)) \\
        $ $   & = & $(a \theta_{_{\alpha, {\alpha + \beta}}}) (b \varphi_{_{\beta, \alpha + \beta}}+ e_{\alpha + \beta})$ \\
        $ $   & = & $(a \theta_{_{\alpha, {\alpha + \beta}}})(b \varphi_{_{\beta, \alpha + \beta}})+ e_{\alpha + \beta}$ \\
        $ $   & = & $(a \theta_{_{\alpha, {\alpha + \beta}}})(b \varphi_{_{\beta, \alpha + \beta}})$\\
        $ $   & = & $(a \Phi_{_{\alpha, \alpha + \beta}})(b \Phi_{_{\beta, \alpha + \beta}})$.
    \end{tabular}
\end{center}

\noindent
Case 7: if $a \in S_{\alpha}$, $b \in R_{\beta}$ then since $ Reg^{+}(S)$ is a bi-ideal we have $a+b \in R_{\alpha + \beta}$ and $ab \in R_{\alpha \beta}$. Then we can prove in a similar way as in Case 6 and obtain $a+b$ = $a \Phi_{_{\alpha, \alpha + \beta}}+b \Phi_{_{\beta, \alpha + \beta}}$ and $(ab) \Phi_{_{\alpha \beta, \alpha+\beta}}$ = $(a \Phi_{_{\alpha, \alpha + \beta}})(b \Phi_{_{\beta, \alpha + \beta}})$.

\noindent
Thus it is proved that $S$ is a strong b-lattice of quasi skew-rings, denoted by $S<Y, T_\alpha, \Phi_{_{\alpha,\beta}}>$ if it satisfies the conditions.

\noindent
Conversely, suppose that $S$ is a strong b-lattice of quasi skew-rings, denoted by $S<Y, T_\alpha, \Phi_{_{\alpha,\beta}}>$. By definition of strong b-lattice of semirings, $S = \displaystyle{\bigcup_{\alpha \in Y}T_\alpha}$, where $Y$ is a b-lattice and $T_{\alpha}$ is quasi skew-ring for each $\alpha \in Y$. Then $S$ is clearly a quasi completely inverse semiring. For each $\alpha \in Y$, $R_\alpha$ is a bi-ideal of $T_\alpha$ \cite{maity2} i.e., $T_\alpha$ is a nil-extension of $R_\alpha$. We recall that any completely regular semiring is a union of (disjoint) skew-rings \cite{maity5}. Therefore by Theorem \ref{2}, $\displaystyle{\bigcup_{\alpha \in Y}R_\alpha} = Reg^{+}(S)$. Now if $a \in Reg^{+}(S)$, $b \in S$, then $a \in R_{\alpha}$ for some $\alpha \in Y$ and we can very clearly show that $a+b, b+a, ab, ba \in Reg^{+}(S)$, i.e., $Reg^{+}(S)$ is a bi-ideal of $S$.
Now, $Reg^{+}(S)$ is clearly an additive inverse subsemiring semiring and completely regular subsemiring of $S$. Therefore $S$ is a strongly additively quasi completely inverse semiring and by \cite[Theorem 1.2.]{howie}, the additive idempotent elements of $Reg^{+}(S)$ commute. This implies that $(Reg^{+}(S), +)$ is a Clifford semigroup and hence by \cite[Theorem 2.1.]{howie}, $(Reg^{+}(S), +)$ is a strong semilattice of groups $(R_{\alpha}, +)$ on the semilattice $(Y,+)$ where the structural description of $\Phi_{_{\alpha,\beta}}$ is injective. Hence from \cite[Theorem 2.5.]{sen1}, clearly $Reg^{+}(S)$ is a generalized Clifford subsemiring of $S$.
Let ${\Phi_{_{_{\alpha, \beta}}}}{\big{|}_{R_{\alpha}}}=\theta_{_{_{\alpha, \beta}}}$. Then $Reg^{+}(S) = R <Y, R_\alpha, \theta_{_{\alpha,\beta}}> $. Let ${\Phi_{_{_{\alpha, \beta}}}}{\big{|}_{S_{\alpha}}}=\varphi_{_{_{\alpha, \beta}}}$. Then it can be easily proved that $\varphi_{_{_{\alpha, \beta}}}$ is a homomorphism from the partial semiring $S_{\alpha}$ to $T_{\beta}$ such that conditions of (ii) and (iii) are satisfied. 

For any $\alpha, \beta, \gamma \in Y$ with $\alpha+\beta \leq \gamma$ and $a \in S_{\alpha}$, $b \in S_\beta$, 

\begin{center}
 \begin{tabular}{ccl}
   $a \varphi_{_{_{\alpha, \gamma}}} + b \varphi_{_{_{\beta, \gamma}}}$   & = &  $a \Phi_{_{_{\alpha, \gamma}}}$ + ~$b \Phi_{_{_{\beta, \gamma}}}$ =  $a \Phi_{_{_{\alpha, \alpha +  \beta}}} \Phi_{_{_{\alpha + \beta, \gamma}}}$ + ~$b \Phi_{_{_{\beta, \alpha + \beta}}} \Phi_{_{_{\alpha + \beta, \gamma}}}$ \\
   $ $ & = &  $(a \Phi_{_{_{\alpha, \alpha + \beta}}} + ~ b \Phi_{_{_{\beta, \alpha + \beta}}}) \Phi_{_{_{\alpha + \beta, \gamma}}}$ = $(a+b) \Phi_{_{_{\alpha + \beta, \gamma}}}$. \\
\end{tabular}   
\end{center}

From above we can clearly see that if $a+b \notin S_{\alpha + \beta}$, $a \varphi_{_{_{\alpha, \gamma}}} + ~b \varphi_{_{_{\beta, \gamma}}} \notin S_{\gamma}$. So Condition (ii) (2) holds.  
Again,

\begin{center}
 \begin{tabular}{ccl}
   $(a \varphi_{_{_{\alpha, \gamma}}}) (b \varphi_{_{_{\beta, \gamma}}})$   & = &  $(a \Phi_{_{_{\alpha, \gamma}}}) (b \Phi_{_{_{\beta, \gamma}}})$ =  $(a \Phi_{_{_{\alpha, \alpha + \beta}}} \Phi_{_{_{\alpha + \beta, \gamma}}})$ $(b \Phi_{_{_{\beta, \alpha + \beta}}} \Phi_{_{_{\alpha + \beta, \gamma}}})$ \\
   $ $ & = &  $(a \Phi_{_{_{\alpha, \alpha + \beta}}} b \Phi_{_{_{\beta, \alpha + \beta}}}) \Phi_{_{_{\alpha + \beta, \gamma}}}$ = $(ab \Phi_{_{_{\alpha \beta, \alpha + \beta}}}) \Phi_{_{_{\alpha + \beta, \gamma}}}$ \\
    $ $ & = &  $ab \Phi_{_{_{\alpha \beta, \gamma}}}$.  
\end{tabular}   
\end{center}

Also, if $ab \notin S_{\alpha \beta}$, $(a \varphi_{_{_{\alpha, \gamma}}}) (b \varphi_{_{_{\beta, \gamma}}}) \notin S_{\gamma}$. So Condition (ii) (4) holds. 

\noindent
It is easy to see that Condition (ii)(1), (3) and (5) will hold from the definition of strong b-lattice.

Moreover, for any $\alpha, \beta \in Y$ with $\alpha \leq \beta$ and $a \in S_{\alpha}$, then 

\begin{center}
    $a \psi \theta_{_{_{\alpha, \beta}}}$ = $(a+e_{\alpha}) \theta_{_{_{\alpha, \beta}}}$ = $(a+e_{\alpha}) \Phi_{_{_{\alpha, \beta}}}$ = $a \Phi_{_{_{\alpha, \beta}}} + e_{\alpha} \Phi_{_{_{\alpha, \beta}}}$ = $a \varphi_{_{_{\alpha, \beta}}} + e_{\beta}$ = $a \varphi_{_{_{\alpha, \beta}}} \psi$. 
\end{center}

This completes the proof.
\end{proof}

\end{document}